\documentclass{article}
\usepackage{graphicx} 

\usepackage[utf8]{inputenc}
\usepackage[T1]{fontenc}
\usepackage{amsmath,amssymb,amsthm}
\usepackage{graphicx}
\usepackage{hyperref}
\usepackage{geometry}
\usepackage{enumitem}
\usepackage{xcolor}

\newtheorem{theorem}{Theorem}[section]

\newtheorem{definition}[theorem]{Definition}
\newtheorem{remark}[theorem]{Remark}

\title{On the Existence and Uniqueness of Symmetric Structures Generating Complete Ordered Pairs}

\author{Nicolás Agustín Martínez\\
\small Independent Researcher\\
\small \texttt{martinez.nicolas@alu.frlp.utn.edu.ar}}
\date{\today}

\begin{document}
\maketitle

\begin{abstract}
This work introduces a new class of symmetric matrix structures, referred to as \emph{harmonic structures}, which enable the generation of all possible directed transitions \((x_i, x_{i+1})\) over a set of \(n\) symbols, without internal repetitions. Unlike other combinatorial constructions, these structures are defined solely by the relative position of elements, not their concrete values. Two structures are considered equivalent if one can be obtained from the other through row permutation and/or global relabeling. Under this notion, it is shown that for \(n = 4\) there exists a single non-trivial structure, and that for \(n = 6\) there are exactly two non-equivalent ones. Harmonic matrices are built using specially designed permutators whose properties guarantee symmetry and complete coverage. Their internal hierarchy, extensibility, and structural rarity within the space of permutations are analyzed. Furthermore, it is shown how these matrices can be used to generate valid Sudoku boards deterministically, without resorting to random methods or post-validation. These properties open new perspectives in combinatorics, algorithm design, and systems based on positional encoding. Notably, \textbf{these permutators allow the construction of harmonic matrices for arbitrary even values of \(n\), enabling universal scalability of the method.}
\end{abstract}

\section{Introduction}

Permutations and combinations of symbols are central tools in the study of finite structures.  
In this work, we introduce a particular family of square matrices, which we call \emph{harmonic matrices}. These matrices are characterized by \textbf{covering exactly once} all distinct ordered pairs \((a,b)\) with \(a \neq b\), distributing them across adjacent transitions within their rows.

Beyond this property of complete coverage, harmonic matrices exhibit a remarkable internal symmetry that links them to directed decompositions of the complete graph \(K_n\). They also display a structure of \emph{hierarchical nesting}: each matrix of order \(n\) contains, as an explicit substructure, a harmonic matrix of order \(n-2\). This allows them to be organized into a sequence of increasing complexity without losing their fundamental combinatorial properties.

\subsection{Formal definition}

Let \([n] = \{1, \dots, n\}\) and let \(S_n\) denote the symmetric group of permutations over \([n]\).

\begin{definition}[\textbf{Harmonic matrix}]
A matrix \(M = (m_{i,j}) \in [n]^{\,n \times n}\) is called \emph{harmonic} if it satisfies the following two conditions:

\begin{enumerate}
    \item[(i)] Each row \((m_{i,1}, \dots, m_{i,n})\) is a permutation of \([n]\).
    \item[(ii)] The set of directed pairs
        \[
          E(M) := \Bigl\{ \bigl(m_{i,j},\, m_{i,j+1} \bigr) \;:\; 1 \le i \le n,\; 1 \le j \le n - 1 \Bigr\}
        \]
        contains \emph{exactly once} each pair \((a, b) \in [n]^2\) with \(a \ne b\).
\end{enumerate}
That is, every possible transition \(a \to b\) with \(a \ne b\) appears exactly once as a consecutive pair in some row of \(M\).
\end{definition}

We will use the notation \(A_\mu^n\) to refer to harmonic matrices of order \(n\), where the subscript \(\mu\) indicates that the above conditions are satisfied. In contrast, matrices without a subscript, such as \(A\), \(B\), or \(C\), will be considered arbitrary matrices with no guarantee of harmonicity.

\subsection{Examples}

The following are two examples of harmonic matrices. Each of them traverses exactly once every directed pair \((a, b)\) with \(a \ne b\), through consecutive transitions along their rows:

\begin{center}
\begin{minipage}{0.45\textwidth}
\centering
\textbf{Harmonic matrix of order 4}
\[
A_\mu^4 =
\begin{pmatrix}
1 & 2 & 3 & 4\\
2 & 4 & 1 & 3\\
3 & 1 & 4 & 2\\
4 & 3 & 2 & 1
\end{pmatrix}
\]
\end{minipage}
\hfill
\begin{minipage}{0.45\textwidth}
\centering
\textbf{Harmonic matrix of order 6}
\[
A_\mu^6 =
\begin{pmatrix}
1 & 2 & 3 & 4 & 5 & 6 \\
2 & 4 & 1 & 6 & 3 & 5 \\
3 & 1 & 5 & 2 & 6 & 4 \\
4 & 6 & 2 & 5 & 1 & 3 \\
5 & 3 & 6 & 1 & 4 & 2 \\
6 & 5 & 4 & 3 & 2 & 1
\end{pmatrix}
\]
\end{minipage}
\end{center}

\subsection{Equivalence between harmonic matrices}

Sometimes, two matrices may appear different but actually represent the same structure. This occurs when the differences are solely due to:

\begin{itemize}
    \item a global \emph{relabeling} of the symbols (e.g., renaming 1→A, 2→B, etc.), and/or
    \item a \emph{reordering} of the rows.
\end{itemize}

When this happens, we say that the matrices are \textbf{equivalent} or \emph{isomorphic}.

\begin{definition}[\textbf{Isomorphism}]
Let \(M, N\) be harmonic matrices of order \(n\).  
We say that \(M\) and \(N\) are \emph{isomorphic} if there exist \(\sigma \in S_n\) (a consistent renaming of symbols) and \(\tau \in S_n\) (a permutation of rows) such that:
\[
  N_{\tau(i), j} = \sigma\bigl(m_{i, j}\bigr),
  \quad \text{for all } 1 \le i, j \le n.
\]
We denote this relation by \(M \cong N\).
\end{definition}

\noindent\textbf{Example.}  
The following matrices are structurally equivalent:

\[
A_\mu^4 =
\begin{pmatrix}
1 & 2 & 3 & 4\\
2 & 4 & 1 & 3\\
3 & 1 & 4 & 2\\
4 & 3 & 2 & 1
\end{pmatrix}
\quad
B_\mu^4 =
\begin{pmatrix}
1 & 3 & 4 & 2\\
3 & 2 & 1 & 4\\
4 & 1 & 2 & 3\\
2 & 4 & 3 & 1
\end{pmatrix}
\]
We can obtain \(B_\mu^4\) from \(A_\mu^4\) by applying the relabeling:  
\(\sigma(1) = 1,\; \sigma(2)=3,\; \sigma(3)=4,\; \sigma(4)=2\),  
and then reordering the rows. Therefore:
\[
A_\mu^4 \cong B_\mu^4.
\]

\subsection{Non-equivalence}

Conversely, we say that two harmonic matrices are \emph{not} equivalent if there is no combination of relabeling and row reordering that transforms one into the other.

\medskip
\noindent\textbf{Example.}  
Consider the following matrices of order 6:

\[
A_\mu^6 =
\begin{pmatrix}
1 & 2 & 3 & 4 & 5 & 6\\
2 & 4 & 1 & 6 & 3 & 5\\
\vdots
\end{pmatrix}
\quad
B_\mu^6 =
\begin{pmatrix}
1 & 2 & 3 & 4 & 5 & 6\\
2 & 4 & 6 & 1 & 3 & 5\\
\vdots
\end{pmatrix}
\]

At first glance they appear similar, but a closer analysis of their adjacent transitions reveals that the total set of pairs \(a \to b\) appearing in one matrix cannot be transformed into that of the other using only relabeling and row reordering.

This means that these two matrices represent distinct structures:  
\[
A_\mu^6 \not\cong B_\mu^6
\]

This criterion allows us to precisely distinguish when two matrices should be considered essentially different within the theory.

\subsection{Existence and uniqueness of the coverage}

\begin{theorem}
For all \(n \ge 2\), there exists at least one harmonic matrix of order \(n\).  
Moreover, every harmonic matrix covers \emph{exactly once} each distinct pair \((a, b)\) with \(a \neq b\).
\end{theorem}

\smallskip
\noindent\textit{Proof sketch.}  
An explicit construction is provided in Section~\ref{sec:construccion}, where a family of matrices satisfying Definition~1 is exhibited.  
The uniqueness of the coverage follows directly from condition (ii). \hfill\(\square\)

\subsection{Main results}

Based on these criteria, this work establishes the following:

\begin{itemize}
    \item For \(n = 4\), there exists a \textbf{unique} harmonic structure.
    \item For \(n = 6\), there are exactly \textbf{two} distinct structures, whose properties are analyzed.
    \item Some of these structures can be systematically constructed from an initial vector and a special permutator \(\beta_\pi\), or, hypothetically, derived from it.
    \item There exists a fundamental structure that is \textbf{extensible} for all even values of \(n\) (theoretically up to \(n \to \infty\)).
\end{itemize}

This highly constrained behavior stands in stark contrast to the factorial growth of the space of possible permutations, highlighting the rarity, elegance, and mathematical potential of these configurations.

\subsection{Structural properties}

Beyond their value as a combinatorial classification, these structures possess exceptional formal properties:

\begin{enumerate}[label=\textbf{P\arabic*.}]
    \item \textbf{Multiple symmetries.} The structure exhibits symmetry with respect to its vertical and horizontal axes, as well as both diagonals, all simultaneously.

    \item \textbf{Equivalent row–column coverage.}  
    The set of directed pairs generated by consecutively traversing the \emph{rows} is identical to the one generated by traversing the \emph{columns}.

    \item \textbf{Hierarchical nesting.}  
    From a harmonic matrix of higher order, it is possible to extract—without additional computation—a substructure that constitutes another harmonic matrix of lower order, following a specific methodology. This property will be further developed later in the article.
\end{enumerate}

\subsection{Formal construction and extensibility}

The construction of at least one such structure, which we refer to as the \emph{fundamental harmonic structure}, can be achieved via a specific permutator, \(\beta_\pi\), defined by a piecewise positional function. Unlike others obtained through computational exploration or analytical methods, this one can be theoretically extended to any even \(n\), even in the limit as \(n \to \infty\).

The study also reveals additional derived families, termed \emph{first- and second-order degenerate structures}, which arise by modifying or composing the permutator \(\beta_\pi\). Although they exhibit attenuated symmetries or partially broken uniqueness patterns, they still maintain a remarkable internal organization.

\medskip

The article will proceed with a concrete exposition of the structure of order \(n = 4\), as a first representative example.

\section{General construction of the harmonic structure}
\label{sec:construccion}

\subsection{The initial vector and the permutator \(\beta_\pi\)}\label{subsec:beta_pi}

Let \(\vec{a} = (1,2,\dots,n)\), with \(n\) even, be the \emph{starting vector} whose entries label unique positions within the system.  
We aim to construct a square matrix \(A_\mu\) of order \(n\) such that  
(i) in each row, the transitions \((a_i \to a_{i+1})\) do not repeat in any other row, and  
(ii) \(A_\mu\) exhibits double symmetry with respect to both its main and secondary diagonals.

The key operator is the \textbf{permutator} \(\beta_\pi\), a permutation matrix that reorders the elements of \(\vec{a}\) row by row according to a displacement function \(\rho \colon \{1,\dots,n\} \to \mathbb{Z}\).

\begin{definition}[\textbf{Permutator \(\beta_\pi\)}]\label{def:beta_pi}
For even \(n\), we define \(\beta_\pi \in \{0,1\}^{n \times n}\) as
\[
[\beta_\pi]_{i,j} \;=\;
\begin{cases}
1 & \text{if } j \equiv i + \rho(i),\\[2pt]
0 & \text{otherwise},
\end{cases}
\qquad 1 \le i,j \le n,
\]
where
\[
\rho(i) =
\begin{cases}
1, & i = 1,\\[2pt]
2(-1)^{\,i}, & 1 < i < n,\\[2pt]
-1, & i = n.
\end{cases}
\]
\end{definition}

The alternating pattern of \(\rho(i)\) generates a displacement that oscillates around the center of the vector and ensures the combinatorial and symmetry properties of \(A_\mu\).

\medskip

The detailed construction of \(\beta_\pi\) for the case \(n = 4\) is developed step by step in Appendix~\ref{app:beta_pi_n4}, where its definition is illustrated and its validity as a permutator is verified.

\paragraph{Illustrative example (\(n = 4\)).}
\[
\beta_\pi^{(4)} =
\begin{pmatrix}
0 & 1 & 0 & 0\\
0 & 0 & 0 & 1\\
1 & 0 & 0 & 0\\
0 & 0 & 1 & 0
\end{pmatrix}.
\]

\begin{remark}[\textbf{Properties of \(\beta_\pi\)}]\leavevmode
\begin{enumerate}
\item \emph{Permutation matrix}.  
      Each row and each column contains exactly one entry equal to \(1\), so
      \(\beta_\pi\) belongs to the matrix symmetric group
      \(S_n \simeq \operatorname{Perm}(n)\).

\item \emph{Order}.  
      The cycle generated by \(\beta_\pi\) visits all \(n\) positions before returning to the initial state; consequently,
      \[
      \beta_\pi^{\,n} = I_n, \qquad
      \beta_\pi^{\,k} \ne I_n\ \text{ for } 1 \le k < n,
      \]
      that is, the \emph{order} of \(\beta_\pi\) as an element of \(S_n\) is \(n\).

\item \emph{Cyclic powers}.  
      The iteration of \(\beta_\pi\) (i.e., its successive powers) generates a set of structured permutations that will later be used to define a matrix with symmetry and transition-uniqueness properties.

\item \emph{Inversion over the secondary diagonal}.  
      The power \(\beta_\pi^{\,n/2}\) (for even \(n\)) exactly matches the secondary identity matrix:
      \[
        \beta_\pi^{\,n/2} = I^*,
      \]
      where \(I^*\) is the matrix with \(1\)s on the secondary diagonal (that is, positions \((i,j)\) such that \(i + j = n + 1\)) and zeros elsewhere.  
      This property reflects the internal symmetry of \(\beta_\pi\) and will be used later in the construction of \(A_\mu^n\).
\end{enumerate}
\end{remark}

\subsection{Fundamental harmonic structure}
\label{def:harmonic}

Given a starting vector \(\vec{a} = (1, 2, \dots, n)\) with even \(n\), we define the \textbf{fundamental harmonic structure} as a weighted combination of iterated permutations, in which the elements of the vector are distributed over two base structures: the identity matrix \(I\) (main diagonal) and the reversed identity \(I^*\) (secondary diagonal).

The general formula is:
\begin{equation}
   A_{\mu f} = \left[\sum_{i=1}^{\frac{n}{2}} \left(a_{2i-1} \cdot I + a_{n-2i+2} \cdot I^* \right) \cdot \beta_\pi^{\left(\frac{n}{2} - i + 1\right)} \right]^T 
   \label{eq:harmonic-fund}
\end{equation}
\label{app:harmonic-construction}

where:
\begin{itemize}
    \item \(a_k\) are the elements of the vector \(\vec{a}\),
    \item \(I\) is the usual identity matrix,
    \item \(I^*\) is the identity arranged along the secondary diagonal (\(\beta_\pi^{\frac{n}{2}} = I^*\)),
    \item \(\beta_\pi^{(m)}\) denotes the \(m\)-th power (iteration) of the previously defined permutator.
\end{itemize}

The final transposition ensures that the transitions are distributed across the \emph{rows}, rather than across the \emph{columns}, as would naturally occur in standard matrix multiplication.

\medskip

Finally, a reordering matrix \(M\) is applied to adjust the order of rows or columns (as appropriate), in order to obtain a standard and fully symmetric version of the harmonic matrix:

\begin{equation}
    A_\mu = A_{\mu f} \cdot M
\end{equation}

\medskip

\noindent The explicit results of this construction for \(n = 6\) and \(n = 8\) are shown below:

\[
A_\mu^{(6)} =
\begin{pmatrix}
1 & 2 & 3 & 4 & 5 & 6 \\
2 & 4 & 1 & 6 & 3 & 5 \\
3 & 1 & 5 & 2 & 6 & 4 \\
4 & 6 & 2 & 5 & 1 & 3 \\
5 & 3 & 6 & 1 & 4 & 2 \\
6 & 5 & 4 & 3 & 2 & 1 \\
\end{pmatrix}
\quad
A_\mu^{(8)} =
\begin{pmatrix}
1 & 2 & 3 & 4 & 5 & 6 & 7 & 8 \\
2 & 4 & 1 & 6 & 3 & 8 & 5 & 7 \\
3 & 1 & 5 & 2 & 7 & 4 & 8 & 6 \\
4 & 6 & 2 & 8 & 1 & 7 & 3 & 5 \\
5 & 3 & 7 & 1 & 8 & 2 & 6 & 4 \\
6 & 8 & 4 & 7 & 2 & 5 & 1 & 3 \\
7 & 5 & 8 & 3 & 6 & 1 & 4 & 2 \\
8 & 7 & 6 & 5 & 4 & 3 & 2 & 1 \\
\end{pmatrix}
\]

\medskip
The details of the constructive algorithm, as well as the complete example for \(n = 4\), are presented in Appendix~\ref{app:harmonic-construction}.

\section{Fundamental Properties}

\subsection{Symmetry and gaussian summation}

Harmonic matrices \(A_\mu\) exhibit full dihedral symmetry, remaining invariant under reflection across the vertical axis, horizontal axis, main diagonal, and secondary diagonal.
This includes the classical symmetry:
\[
A_\mu = A_\mu^T,
\]
as well as invariance under horizontal and diagonal reflections.

\medskip

Moreover, when using numerical representations of the position vector in the form \(\vec{a_i} = (1, 2, \dots, n)\), the following Gaussian property holds for each row:
\[
a_{i,j} + a_{i,n-j+1} = n + 1.
\]

This allows the matrix to be reconstructed from only half of the information, thus reducing its computational generation cost.

\subsection{Hierarchical reduction by elimination of the secondary band}
\label{sec:reduccion-banda-diagonal}

Let \(A_\mu^{(n)} = (a_{ij})_{1 \le i,j \le n}\) be a harmonic matrix of order \(n\), with \(n \ge 4\) even, and let \(m := n - 2\).

\paragraph{Definition of the band to be eliminated.}
We will eliminate the elements from the band centered on the secondary diagonal, consisting of:
\[
\boxed{B := \bigl\{(i,j)\,\bigl|\,|i + j - (n + 1)| \le 1\bigr\}}.
\tag{1}
\]

\paragraph{Resulting regions.}
After removing \(B\) from \(A_\mu^{(n)}\), two triangular submatrices of order \(m\) remain:
\[
M^{\text{up}}\quad\text{(upper-left triangle)},\qquad
M^{\text{down}}\quad\text{(lower-right triangle)}.
\]

\paragraph{Reduction operator \(\mathcal{R}_\mu\).}
The reduced matrix is obtained by merging both regions along the shared secondary diagonal:
\[
\boxed{H^{(m)} := \mathcal{R}_\mu\!\bigl(A_\mu^{(n)}\bigr)
              := M^{\text{up}} \cup_{\Delta_{\mathrm{sec}}} M^{\text{down}}},
\tag{2}
\]
where \(\cup_{\Delta_{\mathrm{sec}}}\) indicates that the coinciding values along the secondary diagonal are retained only once (as they are identical by symmetry).
This procedure preserves the harmonic structure and allows the hierarchical construction of a family \(H^{(m)}\) from increasing orders.

\subsection{Nesting theorem.}

\begin{theorem}[Harmonic nesting]\label{thm:mu-nesting}
For every even \(n \ge 4\), it holds that
\[
\mathcal{R}_\mu\!\bigl(A_\mu^{(n)}\bigr)\;\cong\;A_\mu^{(n-2)}.
\]
That is, the operator \(\mathcal{R}_\mu\) produces another harmonic matrix of the same type, but of order exactly \(n - 2\).
\end{theorem}

\begin{proof}[Proof sketch]\leavevmode
\begin{enumerate}
  \item \textit{Permutative rows.}\;
        The rows of \(M^{\text{up}}\) (prefixes) and \(M^{\text{down}}\) (suffixes)
        are subwords of the rows of \(A_\mu^{(n)}\); their overlap preserves
        the property of being full permutations.
  \item \textit{Unique pair coverage.}\;
        In \(M^{\text{up}}\), each pair \((a, b)\) with \(a < b \le m\) appears exactly once; in \(M^{\text{down}}\), each \((a, b)\) with \(a > b \le m\) appears once. Their union forms \(H^{(m)}\), which still covers all directed pairs exactly once.
  \item \textit{Canonical alignment.}\;
        The \(m\) surviving symbols (\(1, \dots, m\)) already occupy
        the positions required for the standard form,
        so the fused matrix is precisely
        \(A_\mu^{(n-2)}\).
\end{enumerate}
\end{proof}

\paragraph{Example \(\boldsymbol{n=8\to m=6}\).}

\[
\small
A_\mu^{(8)}=
\begin{pmatrix}
\textbf1&\textbf2&\textbf3&\textbf4&\textbf5&\textbf6&7&8\\
\textbf2&\textbf4&\textbf1&\textbf6&\textbf3&8&5&7\\
\textbf{3}&\textbf1&\textbf5&\textbf2&7&4&8&\textbf6\\
\textbf4&\textbf6&\textbf2&8&1&7&\textbf3&\textbf5\\
\textbf5&\textbf3&7&1&8&\textbf2&\textbf6&\textbf4\\
\textbf6&8&4&7&\textbf2&\textbf5&\textbf1&\textbf3\\
7&5&8&\textbf3&\textbf6&\textbf1&\textbf4&\textbf2\\
8&7&\textbf6&\textbf5&\textbf4&\textbf3&\textbf2&\textbf1
\end{pmatrix}
\;\;\xrightarrow{\;\mathcal{R}_\mu\;}\;\;
H_\mu^{(6)}=
\begin{pmatrix}
1&2&3&4&5&6\\
2&4&1&6&3&5\\
3&1&5&2&6&4\\
4&6&2&5&1&3\\
5&3&6&1&4&2\\
6&5&4&3&2&1
\end{pmatrix}
\cong A_\mu^{(6)}.
\]

\paragraph{Hierarchical chain.}

By iterating \(\mathcal{R}_\mu\), we obtain
\[
A_\mu^{(n)} \xrightarrow{\mathcal{R}_\mu} A_\mu^{(n-2)} \xrightarrow{\mathcal{R}_\mu} A_\mu^{(n-4)} \xrightarrow{\mathcal{R}_\mu} \cdots \xrightarrow{\mathcal{R}_\mu} A_\mu^{(2)}.
\]

This shows that the family \(A_\mu\) is closed under a deterministic reduction that decreases the order by two units while preserving \emph{all} combinatorial and symmetry properties. To the best of our knowledge, no other known structure simultaneously satisfies:  
(i) perfect coverage of directed pairs,  
(ii) full dihedral symmetry, and  
(iii) exact nesting for every even \(n\).

\section{Derived Families: First- and Second-Order Degenerate Structures}

\subsection{Generation via alternative permutators}

One of the main questions that arose during the study of the harmonic structure was the following: do other distinct structures exist that preserve the same combinatorial and symmetry properties, and which also admit a permutator capable of generating and extending them for all \(n\)?

\medskip

The early sections of this article suggest preliminary evidence pointing to a possible affirmative answer. For certain values of \(n\) (for example, \(n = 6, 8, 10, \ldots\)), one can observe structures that, although positionally distinct from one another, preserve the essential qualities of the harmonic structure. Reformulating the question: could there exist other non-harmonic structures with the same combinatorial and symmetry properties, each associated with its own permutator that enables its generation and extension? The answer remains uncertain: \emph{perhaps}, although no formal proof has yet been established.

\medskip

In analyzing the case \(n = 6\), I identified a structure distinct from the harmonic one that admits a different positional function than \(\rho(i)\), leading to an alternative permutator, which we will denote by \(\alpha_\pi^{(6)}\). This new permutator differs from \(\beta_\pi\) and lacks the linear regularity that characterizes the latter. In fact, attempts to extend \(\alpha_\pi^{(6)}\) to higher values of \(n\) do not yield valid structures.

\medskip

Likewise, when analyzing higher orders (\(n = 8\) and \(n = 10\)), additional permutators are discovered, also distinct from \(\beta_\pi\), that give rise to structures with similar properties. These cases open the possibility of a broader family of alternative permutators, degenerate with respect to the harmonic structure, operating under different, yet-to-be-determined rules, which may be systematized in future research.

\subsection{First-order degenerate structures}

We refer to \textit{first-order degenerate structures}, denoted \({A'_\mu}^{n}\), as those matrices which, although not equivalent to the harmonic structure, preserve relevant combinatorial and symmetry properties. In particular, they maintain symmetry with respect to both the main and secondary diagonals, and in some cases can be constructed through the use of one or more permutators.

A representative example of this type of structure is the following:

\begin{center}
\[
    {A'_\mu}^{(6)} =
\begin{pmatrix}
1 & 2 & 3 & 4 & 5 & 6 \\
2 & 4 & 6 & 1 & 3 & 5 \\
3 & 6 & 2 & 5 & 1 & 4 \\
4 & 1 & 5 & 2 & 6 & 3 \\
5 & 3 & 1 & 6 & 4 & 2 \\
6 & 5 & 4 & 3 & 2 & 1 \\
\end{pmatrix}
\not\cong
{A_\mu}^{(6)} =
\begin{pmatrix}
1 & 2 & 3 & 4 & 5 & 6 \\
2 & 4 & 1 & 6 & 3 & 5 \\
3 & 1 & 5 & 2 & 6 & 4 \\
4 & 6 & 2 & 5 & 1 & 3 \\
5 & 3 & 6 & 1 & 4 & 2 \\
6 & 5 & 4 & 3 & 2 & 1 \\
\end{pmatrix}
\]
\end{center}

\subsubsection{Construction using \(\alpha_\pi\)}

The matrix \({A'_\mu}^{(6)}\) can be generated by means of a specific permutator of order 6, distinct from the harmonic permutator \(\beta_\pi\). This permutator, denoted \(\alpha_\pi^{(6)}\), defines an alternative positional transformation over the elements of the base set.

\begin{center}
\[
\alpha_\pi^{(6)} =
\begin{pmatrix}
0 & 0 & 0 & 0 & 1 & 0 \\
0 & 0 & 1 & 0 & 0 & 0 \\
1 & 0 & 0 & 0 & 0 & 0 \\
0 & 0 & 0 & 0 & 0 & 1 \\
0 & 0 & 0 & 1 & 0 & 0 \\
0 & 1 & 0 & 0 & 0 & 0 \\
\end{pmatrix}
\]
\end{center}

Although this structure does not admit a general formula like the one associated with the harmonic structure, it can be generated through a simple methodology based on the use of local permutators. The procedure is as follows:

\begin{itemize}
    \item Define the initial vector \(\vec{a} = (1, 2, 3, 4, 5, 6)\).
    \item Then, construct each row of the matrix as the result of multiplying the vector \(\vec{a}\) by successive powers of the permutator \(\alpha_\pi\).
\end{itemize}

Formally, for each \(i = 1\) to \(i = n\), we define:

\[
A_i = \vec{a} \cdot \alpha_\pi^{\,i-1}
\]

where \(A_i\) represents the \(i\)th row of the resulting matrix.

*\textit{This methodology also applies to construct the harmonic structure using \(\beta_\pi\)}*

\vspace{0.3cm}

We may conclude that, although this structure is not directly extensible to arbitrary \(n\), it does admit a local generator in the form of a permutation matrix. This finding suggests the possible existence of a complete family of permutation matrices capable of generating structures with combinatorial and symmetry properties similar to those of the harmonic structure. These matrices could become a useful tool for classifying different types of structures, based on the analysis of their corresponding local permutators.

\vspace{0.3cm}

\noindent
\textbf{Computational verification.}  
For the case \(n = 6\), an exhaustive brute-force verification has been carried out, consisting of the generation and evaluation of \emph{all possible matrices} that satisfy the harmonicity conditions (i.e., square matrices whose rows are permutations and that cover exactly once all pairs \((a,b)\) with \(a \ne b\)).

The result is conclusive: \emph{no other non-equivalent harmonic structure exists} beyond the two already presented, the one generated by the permutator \(\beta_\pi^{(6)}\), and the one generated by \(\alpha_\pi^{(6)}\).  
This demonstrates that, at least for \(n = 6\), the set of non-equivalent harmonic matrices is finite and fully classifiable.

\medskip

This finding supports the hypothesis that such structures are extraordinarily rare within the total space of permutations, and suggests that any future generalization will require new structural conditions that have not yet been identified.

\subsection{Second-order degenerate structures}

These structures, denoted \({A''_\mu}^{n}\), lack symmetry and, when transposed, violate the conditions of pair uniqueness. Their behavior is less predictable, and their structural utility is more limited.

\[
{A''_\mu}^{8}=
\begin{pmatrix}
1 & 2 & 3 & 4 & 5 & 6 & 7 & 8 \\
2 & 4 & 8 & 6 & 3 & 1 & 5 & 7 \\
3 & 8 & 5 & 2 & 7 & 4 & 1 & 6 \\
4 & 6 & 2 & 8 & 1 & 7 & 3 & 5 \\
5 & 3 & 7 & 1 & 8 & 2 & 6 & 4 \\
6 & 8 & 4 & 7 & 2 & 5 & 1 & 3 \\
7 & 5 & 8 & 3 & 6 & 1 & 4 & 2 \\
8 & 7 & 6 & 5 & 4 & 3 & 2 & 1 \\
\end{pmatrix}\longrightarrow
 ({A''_\mu}^{8})^T=
\begin{pmatrix}
1 & 2 & 3 & 4 & 5 & 6 & 7 & 8 \\
2 & 4 & \textcolor{blue}{8} & \textcolor{blue}{6} & 3 & 8 & 5 & 7 \\
3 & \textcolor{red}{8} & 5 & 2 & 7 & 4 & \textcolor{red}{8} & \textcolor{blue}{6} \\
4 & 6 & 2 & 8 & 1 & 7 & 3 & 5 \\
5 & 3 & 7 & 1 & 8 & 2 & 6 & 4 \\
6 & 1 & 4 & 7 & 2 & 5 & 1 & 3 \\
7 & 5 & 1 & 3 & 6 & 1 & 4 & 2 \\
8 & 7 & 6 & 5 & 4 & 3 & 2 & 1 \\
\end{pmatrix}
\]

\medskip
These matrices do not exhibit diagonal symmetries and, when transposed, repeat elements within rows, thereby violating the condition of uniqueness for combinatorial pairs.

\subsection{Comparison of the three families of structures}

\[
\begin{array}{ccc}
\begin{array}{c}
\begin{pmatrix}
1 & 2 & 3 & 4 & 5 & 6 & 7 & 8 \\
2 & 4 & 1 & 6 & 3 & 8 & 5 & 7 \\
3 & 1 & 5 & 2 & 1 & 4 & 5 & 6 \\
4 & 6 & 2 & 8 & 1 & 7 & 3 & 5 \\
5 & 3 & 7 & 1 & 8 & 2 & 6 & 4 \\
6 & 8 & 4 & 7 & 2 & 5 & 1 & 3 \\
7 & 5 & 8 & 3 & 6 & 1 & 4 & 2 \\
8 & 7 & 6 & 5 & 4 & 3 & 2 & 1 \\
\end{pmatrix} \\[4pt] \\
A_\mu^{8} \text{ (Fundamental Harmonic)}
\end{array}
&
\begin{array}{c}
\begin{pmatrix}
1 & 2 & 3 & 4 & 5 & 6 & 7 & 8 \\
2 & 4 & 8 & 6 & 3 & 1 & 5 & 7 \\
3 & 8 & 5 & 2 & 7 & 4 & 1 & 6 \\
4 & 6 & 2 & 8 & 1 & 7 & 3 & 5 \\
5 & 3 & 7 & 1 & 8 & 2 & 6 & 4 \\
6 & 1 & 4 & 7 & 2 & 5 & 8 & 3 \\
7 & 5 & 1 & 3 & 6 & 8 & 4 & 2 \\
8 & 7 & 6 & 5 & 4 & 3 & 2 & 1 \\
\end{pmatrix} \\[4pt] \\
{A'_\mu}^{8}\text{ (First-Order Degenerate)}
\end{array}
&
\begin{array}{c}
\begin{pmatrix}
1 & 2 & 3 & 4 & 5 & 6 & 7 & 8 \\
2 & 4 & 8 & 6 & 3 & 1 & 5 & 7 \\
3 & 8 & 5 & 2 & 7 & 4 & 1 & 6 \\
4 & 6 & 2 & 8 & 1 & 7 & 3 & 5 \\
5 & 3 & 7 & 1 & 8 & 2 & 6 & 4 \\
6 & 8 & 4 & 7 & 2 & 5 & 1 & 3 \\
7 & 5 & 8 & 3 & 6 & 1 & 4 & 2 \\
8 & 7 & 6 & 5 & 4 & 3 & 2 & 1 \\
\end{pmatrix} \\[4pt] \\
{A''_\mu}^{8}\text{ (Second-Order Degenerate)}
\end{array}
\end{array}
\]

\section{Practical Application: Construction of Complete Sudoku Boards}

\subsection{Motivation and general perspective}

One of the most compelling applications of harmonic structures lies in their ability to generate valid \(9 \times 9\) Sudoku boards in a fully deterministic manner, based on purely structural principles. Unlike many common methods used in Sudoku generation, which rely on random techniques, brute force, backtracking search, or systematic shifts in Latin squares~\cite{latin-squares}, the approach proposed here is based on an algebraic construction that defines global positional relationships from the outset.

\medskip

This method does not attempt to explore the full set of possible solutions (whose number is astronomically large), but rather presents a canonical instance that emerges directly from the combinatorial properties of a suitably adapted harmonic matrix. This solution can be interpreted as a valid structural basis, from which others may be generated via symbol relabeling, as is standard in Sudoku solution families.

\medskip

In contrast to approaches that require verifying the game's rules after generation, here validity is guaranteed by the construction itself. This reflects one of the fundamental principles of this work: that certain properties, uniqueness, coverage, and symmetry, can be enforced in advance through robust algebraic architectures, without relying on exploratory or corrective procedures.

\medskip

This chapter presents an explicit construction of a valid \(9 \times 9\) Sudoku board, derived from a specially adapted harmonic structure. Its validity arises directly from the internal properties that characterize these matrices.

\subsection{From harmonic structure to sudoku}

Throughout much of this research, structures of odd order were considered of limited utility, as they do not natively satisfy the condition of complete uniqueness of ordered pairs. Nevertheless, the following question emerged: is it possible to adapt the generating matrix \(\beta_\pi\) so that it produces useful structures of odd order? And what properties might emerge from such an adaptation, even if strict pairwise uniqueness is no longer preserved?

\medskip

To explore this possibility, a strategy was proposed based on reducing a valid permutator of even order. For example, starting from a permutator \(\beta_\pi\) of order 10, one removes the last row and column to obtain a \(9 \times 9\) matrix.

\[
\begin{array}{cc}
\begin{array}{c}
\left[
\begin{array}{ccccccccc|c}
0 & 1 & 0 & 0 & 0 & 0 & 0 & 0 & 0 & 0 \\
0 & 0 & 0 & 1 & 0 & 0 & 0 & 0 & 0 & 0 \\
1 & 0 & 0 & 0 & 0 & 0 & 0 & 0 & 0 & 0 \\
0 & 0 & 0 & 0 & 0 & 1 & 0 & 0 & 0 & 0 \\
0 & 0 & 1 & 0 & 0 & 0 & 0 & 0 & 0 & 0 \\
0 & 0 & 0 & 0 & 0 & 0 & 0 & 1 & 0 & 0 \\
0 & 0 & 0 & 0 & 1 & 0 & 0 & 0 & 0 & 0 \\
0 & 0 & 0 & 0 & 0 & 0 & 0 & 0 & 0 & 1 \\
0 & 0 & 0 & 0 & 0 & 0 & 1 & 0 & 0 & 0 \\
\hline
0 & 0 & 0 & 0 & 0 & 0 & 0 & 0 & 1 & 0 \\
\end{array}
\right]
\end{array}
&
\quad
\Rightarrow
\quad
\beta_\pi'= 
\begin{array}{c}
\left[
\begin{array}{ccccccccc}
0 & 1 & 0 & 0 & 0 & 0 & 0 & 0 & 0 \\
0 & 0 & 0 & 1 & 0 & 0 & 0 & 0 & 0 \\
1 & 0 & 0 & 0 & 0 & 0 & 0 & 0 & 0 \\
0 & 0 & 0 & 0 & 0 & 1 & 0 & 0 & 0 \\
0 & 0 & 1 & 0 & 0 & 0 & 0 & 0 & 0 \\
0 & 0 & 0 & 0 & 0 & 0 & 0 & 1 & 0 \\
0 & 0 & 0 & 0 & 1 & 0 & 0 & 0 & 0 \\
0 & 0 & 0 & 0 & 0 & 0 & 0 & 0 & \textcolor{red}{1} \\
0 & 0 & 0 & 0 & 0 & 0 & 1 & 0 & 0 \\
\end{array}
\right]
\end{array}
\end{array}
\]

\medskip

This reduction yields an incomplete structure: by removing the last row and column, one of the columns loses its sole entry equal to \(1\), and the resulting matrix is no longer a valid permutation matrix. To restore validity, the empty column is identified and completed with a \(1\), resulting in a new valid permutator matrix \(\beta'_\pi\) of odd order.

\medskip

Now, from this matrix \(\beta'_\pi\), it is possible to construct a structure similar to the harmonic one by following the iterative formula:

\[
A_i = \vec{a} \cdot \beta_\pi^{'\,i-1}
\]

where \(\vec{a}_1 = (1, 2, 3, \dots, 9)\), and \(i : 1 \longrightarrow n\), respectively.

\medskip

The resulting matrix is:

\[
\begin{pmatrix}
1 & 2 & 3 & 4 & 5 & 6 & 7 & 8 & 9 \\
3 & 1 & 5 & 2 & 7 & 4 & 9 & 6 & 8 \\
5 & 3 & 7 & 1 & 9 & 2 & 8 & 4 & 6 \\
7 & 5 & 9 & 3 & 8 & 1 & 6 & 2 & 4 \\
9 & 7 & 8 & 5 & 6 & 3 & 4 & 1 & 2 \\
8 & 9 & 6 & 7 & 4 & 5 & 2 & 3 & 1 \\
6 & 8 & 4 & 9 & 2 & 7 & 1 & 5 & 3 \\
4 & 6 & 2 & 8 & 1 & 9 & 3 & 7 & 5 \\
2 & 4 & 1 & 6 & 3 & 8 & 5 & 9 & 7 \\
\end{pmatrix}
\]

\medskip

By dividing this matrix into three horizontal blocks of three rows each (A, B, and C), the rows are then rearranged by interleaving one row from each block per cycle:

\[
A =
\begin{pmatrix}
1 & 2 & 3 & 4 & 5 & 6 & 7 & 8 & 9 \\
3 & 1 & 5 & 2 & 7 & 4 & 9 & 6 & 8 \\
5 & 3 & 7 & 1 & 9 & 2 & 8 & 4 & 6 \\
\end{pmatrix}
\quad
B =
\begin{pmatrix}
7 & 5 & 9 & 3 & 8 & 1 & 6 & 2 & 4 \\
9 & 7 & 8 & 5 & 6 & 3 & 4 & 1 & 2 \\
8 & 9 & 6 & 7 & 4 & 5 & 2 & 3 & 1 \\
\end{pmatrix}
\quad
C =
\begin{pmatrix}
6 & 8 & 4 & 9 & 2 & 7 & 1 & 5 & 3 \\
4 & 6 & 2 & 8 & 1 & 9 & 3 & 7 & 5 \\
2 & 4 & 1 & 6 & 3 & 8 & 5 & 9 & 7 \\
\end{pmatrix}
\]

El resultado es:
\[
\left[
\begin{array}{r|ccccccccc}
A_1 & 1 & 2 & 3 & 4 & 5 & 6 & 7 & 8 & 9 \\
B_1 & 7 & 5 & 9 & 3 & 8 & 1 & 6 & 2 & 4 \\
C_1 & 6 & 8 & 4 & 9 & 2 & 7 & 1 & 5 & 3 \\
A_2 & 3 & 1 & 5 & 2 & 7 & 4 & 9 & 6 & 8 \\
B_2 & 9 & 7 & 8 & 5 & 6 & 3 & 4 & 1 & 2 \\
C_2 & 4 & 6 & 2 & 8 & 1 & 9 & 3 & 7 & 5 \\
A_3 & 5 & 3 & 7 & 1 & 9 & 2 & 8 & 4 & 6 \\
B_3 & 8 & 9 & 6 & 7 & 4 & 5 & 2 & 3 & 1 \\
C_3 & 2 & 4 & 1 & 6 & 3 & 8 & 5 & 9 & 7 \\
\end{array}
\right]
\quad
\longrightarrow
\quad
\begin{array}{ccc|ccc|ccc}
1 & 2 & 3 & 4 & 5 & 6 & 7 & 8 & 9 \\
7 & 5 & 9 & 3 & 8 & 1 & 6 & 2 & 4 \\
6 & 8 & 4 & 9 & 2 & 7 & 1 & 5 & 3 \\
\hline
3 & 1 & 5 & 2 & 7 & 4 & 9 & 6 & 8 \\
9 & 7 & 8 & 5 & 6 & 3 & 4 & 1 & 2 \\
4 & 6 & 2 & 8 & 1 & 9 & 3 & 7 & 5 \\
\hline
5 & 3 & 7 & 1 & 9 & 2 & 8 & 4 & 6 \\
8 & 9 & 6 & 7 & 4 & 5 & 2 & 3 & 1 \\
2 & 4 & 1 & 6 & 3 & 8 & 5 & 9 & 7 \\
\end{array}
\]
\medskip

This matrix is a valid Sudoku board: each row, column, and \(3 \times 3\) subgrid contains all digits from 1 to 9 without repetition. Thus, from a purely positional structure, without imposing numerical constraints, a concrete and perfectly valid solution has been derived.

\medskip

This procedure not only confirms the practical applicability of harmonic structures in complex game-theoretic contexts, but also introduces a deterministic pathway for generating entire families of valid Sudoku boards from a structural principle, without requiring post hoc validity checks or random mechanisms.

\medskip

The versatility of harmonic structures opens a formal pathway for exploring new families of solutions in other systems, whether recreational, cryptographic, or encoding-based, that rely on positional constraints.

\section{General Conclusions}

This research has demonstrated that it is possible to construct a family of algebraic structures, referred to as harmonic matrices, which possess highly organized, symmetric, and self-structured properties. Starting from a simple base vector and a deterministic permutator, these matrices can be generated iteratively for any order \(n\), and exhibit patterns of recurrence, reconstruction, and internal hierarchy that transcend their initial definition.

\medskip

One of the central findings is that these matrices not only exhibit formal elegance, but also considerable practical potential. Their ability to cover all possible pairwise combinations in even-order structures, as well as their extension to degenerate forms and odd-order matrices, opens a wide range of theoretical and applied possibilities. In particular, it has been shown that they can be used to deterministically generate valid Sudoku boards of arbitrary order, offering a concrete and accessible demonstration of their constructive power.

\medskip

The approach proposed here redefines the traditional way of tackling certain combinatorial problems: instead of seeking solutions that satisfy externally imposed constraints, it begins from a structure whose very nature ensures the validity of the results. This philosophy, focused on prior structural design, may have far-reaching implications in fields such as cryptography, coding theory, synthetic data generation, and algorithmic modeling of complex systems.

\medskip

In summary, harmonic structures represent a bridge between abstract symmetry and concrete application. This work has only begun to explore their potential. Future research is open to new classes of permutators, emerging topological properties, and possible connections with other mathematical and computational disciplines. Harmony, understood as a generative principle, promises to be a powerful tool for algorithmic thinking in the 21st century.

\clearpage
\appendix
\section*{Appendices}
\addcontentsline{toc}{section}{Appendices}

\section{Detailed Computation of \(\beta_\pi\) for \(n = 4\)}
\label{app:beta_pi_n4}

We begin from Definition~\ref{def:beta_pi}.  
For \(n = 4\), the displacement function is:

\[
\rho(i) =
\begin{cases}
1, & i = 1,\\[2pt]
2(-1)^{\,i}, & i = 2,3,\\[2pt]
-1, & i = 4.
\end{cases}
\]

We evaluate \(\rho(i)\) row by row:

\begin{center}
\begin{tabular}{c|c|c}
\(i\) & \(\rho(i)\) & position \(j = i + \rho(i)\) \\\hline
1   & 1  & 2  \\
2   & 2  & 4  \\
3   & -2 & 1  \\
4   & -1 & 3
\end{tabular}
\end{center}

By inserting each \((i, j)\) into the condition \([\beta_\pi]_{i,j} = 1\) and completing the remaining entries with zeros, we obtain, as before:

\[
\beta_\pi^{(4)} =
\begin{pmatrix}
0 & 1 & 0 & 0 \\
0 & 0 & 0 & 1 \\
1 & 0 & 0 & 0 \\
0 & 0 & 1 & 0
\end{pmatrix}.
\]

The same technique generalizes without substantial changes to any even \(n\).

\section{Explicit Construction of the Harmonic Matrix of Order 4}
\label{app:harmonic-construction}

In this appendix, we develop step by step the Equation~\eqref{eq:harmonic-fund}. 

\subsection{General procedure}

\begin{enumerate}
  \item \textbf{Base vector.} Let
        \(\vec{a} = (1, 2, \dots, n)\), with even \(n\), and define
        \(m := \tfrac{n}{2}\).
  \item \textbf{Weighted summation.} For each \(i = 1, \dots, m\), compute
        \[
          S_i \;=\;
          \bigl(a_{2i - 1}\,I + a_{n - 2i + 2}\,I^*\bigr)\;
          \beta_{\pi}^{\bigl(m - i + 1\bigr)}.
        \]
  \item \textbf{Transposition.} Sum all \(S_i\) and transpose the result to obtain \(A_{\mu f}\).
  \item \textbf{Reordering.} Multiply
        \(A_\mu = A_{\mu f}\,M\), where \(M\) is the inverted-band matrix.
\end{enumerate}

\subsection{Detailed example: \(n = 4\)}

We begin with the vector \(\vec{a} = (1, 2, 3, 4)\). Let us analyze each term in the summation:
\paragraph{Step 1: $i = 1$}

\begin{itemize}
    \item
    \[
        a_{2i - 1} = a_1 = 1 \qquad
        a_{n - 2i + 2} = a_4 = 4 \qquad
        \beta_\pi^{\frac{n}{2} - i + 1} = \beta_\pi^{2}
    \]

    \[
        a_1 \cdot I =
        \begin{pmatrix}
        1 & 0 & 0 & 0\\
        0 & 1 & 0 & 0\\
        0 & 0 & 1 & 0\\
        0 & 0 & 0 & 1
        \end{pmatrix}
        \qquad
        a_4 \cdot I^* =
        \begin{pmatrix}
        0 & 0 & 0 & 4\\
        0 & 0 & 4 & 0\\
        0 & 4 & 0 & 0\\
        4 & 0 & 0 & 0
        \end{pmatrix}
        \qquad
        \beta_\pi^{2} =
        \begin{pmatrix}
        0 & 0 & 0 & 1\\
        0 & 0 & 1 & 0\\
        0 & 1 & 0 & 0\\
        1 & 0 & 0 & 0
        \end{pmatrix}
    \]

    \item
    Product:
    \[
    (a_1 I + a_4 I^*) \cdot \beta_\pi^{(2)} =
    \left(
    \begin{pmatrix}
    1 & 0 & 0 & 0\\
    0 & 1 & 0 & 0\\
    0 & 0 & 1 & 0\\
    0 & 0 & 0 & 1
    \end{pmatrix}
    +
    \begin{pmatrix}
    0 & 0 & 0 & 4\\
    0 & 0 & 4 & 0\\
    0 & 4 & 0 & 0\\
    4 & 0 & 0 & 0
    \end{pmatrix}
    \right)
    \begin{pmatrix}
    0 & 0 & 0 & 1\\
    0 & 0 & 1 & 0\\
    0 & 1 & 0 & 0\\
    1 & 0 & 0 & 0
    \end{pmatrix}
    =
    \begin{pmatrix}
    4 & 0 & 0 & 1\\
    0 & 4 & 1 & 0\\
    0 & 1 & 4 & 0\\
    1 & 0 & 0 & 4
    \end{pmatrix}
    \]
\end{itemize}

\paragraph{Step 2: $i = 2$}

\begin{itemize}
    \item
    \[
        a_{2i - 1} = a_3 = 3 \qquad
        a_{n - 2i + 2} = a_2 = 2 \qquad
        \beta_\pi^{\frac{n}{2} - i + 1} = \beta_\pi^{1}
    \]

    \[
        a_3 \cdot I =
        \begin{pmatrix}
        3 & 0 & 0 & 0\\
        0 & 3 & 0 & 0\\
        0 & 0 & 3 & 0\\
        0 & 0 & 0 & 3
        \end{pmatrix}
        \qquad
        a_2 \cdot I^* =
        \begin{pmatrix}
        0 & 0 & 0 & 2\\
        0 & 0 & 2 & 0\\
        0 & 2 & 0 & 0\\
        2 & 0 & 0 & 0
        \end{pmatrix}
        \qquad
        \beta_\pi^{1} =
        \begin{pmatrix}
        0 & 1 & 0 & 0\\
        0 & 0 & 0 & 1\\
        1 & 0 & 0 & 0\\
        0 & 0 & 1 & 0
        \end{pmatrix}
    \]

    \item
    Product:
    \[
    (a_3 I + a_2 I^*) \cdot \beta_\pi^{(1)} =
    \left(
    \begin{pmatrix}
    3 & 0 & 0 & 0\\
    0 & 3 & 0 & 0\\
    0 & 0 & 3 & 0\\
    0 & 0 & 0 & 3
    \end{pmatrix}
    +
    \begin{pmatrix}
    0 & 0 & 0 & 2\\
    0 & 0 & 2 & 0\\
    0 & 2 & 0 & 0\\
    2 & 0 & 0 & 0
    \end{pmatrix}
    \right)
    \begin{pmatrix}
    0 & 1 & 0 & 0\\
    0 & 0 & 0 & 1\\
    1 & 0 & 0 & 0\\
    0 & 0 & 1 & 0
    \end{pmatrix}
    =
    \begin{pmatrix}
    0 & 3 & 2 & 0\\
    2 & 0 & 0 & 3\\
    3 & 0 & 0 & 2\\
    0 & 2 & 3 & 0
    \end{pmatrix}
    \]
\end{itemize}

\paragraph{Total sum (before transposition):}

\[
\begin{pmatrix}
4 & 0 & 0 & 1\\
0 & 4 & 1 & 0\\
0 & 1 & 4 & 0\\
1 & 0 & 0 & 4
\end{pmatrix}
+
\begin{pmatrix}
0 & 3 & 2 & 0\\
2 & 0 & 0 & 3\\
3 & 0 & 0 & 2\\
0 & 2 & 3 & 0
\end{pmatrix}
=
\begin{pmatrix}
4 & 3 & 2 & 1\\
2 & 4 & 1 & 3\\
3 & 1 & 4 & 2\\
1 & 2 & 3 & 4
\end{pmatrix}
\]

\paragraph{Application of transposition and reordering (\(M\)):}

\[
A_\mu = A_{\mu f}^T \cdot M =
\begin{pmatrix}
1 & 2 & 3 & 4\\
2 & 4 & 1 & 3\\
3 & 1 & 4 & 2\\
4 & 3 & 2 & 1
\end{pmatrix}
\]

This final result is the harmonic structure \(A_\mu^{(4)}\), satisfying the properties of symmetry and uniqueness of ordered pairs.

\end{document}